\documentclass{elsarticle}

\usepackage{lineno,hyperref}
\usepackage{amsfonts}
\usepackage[english]{babel}
\usepackage{graphicx,epstopdf,epsfig}
\usepackage{textcomp}
\usepackage[T1]{fontenc}
\usepackage[tone]{tipa}
\usepackage{epsfig,fancyhdr,graphics,amsmath,amssymb,amsthm}
\usepackage {xcolor}
\modulolinenumbers[5]

\journal{Journal of \LaTeX\ Templates}

\newtheorem{definition}{Definition}[section]
\newtheorem{theorem}{Theorem}[section]
  \newtheorem{corollary}{Corollary}[section]
 \newtheorem{lemma}{Lemma}[section] 
  
  \newtheorem{proposition}{Proposition}[section]
  \newtheorem{example}{Example}[section]
    \newtheorem{rk}{Remark}[section]

\newcommand{\diag}{\mathop{\rm diag}\nolimits}

\newcommand {\Int}{\mathop{\rm int}\nolimits}

\providecommand{\norm}[1]{\lVert#1\rVert}
\begin{document}

\begin{frontmatter}

\title{ Algebraic and Geometric Properties of $\mathcal{L}^n_+$-Semipositive Matrices and  $\mathcal{L}^n_+$-Semipositive Cones}

\author{Aritra Narayan Hisabia} 
\author{Manideepa Saha \corref{mycorrespondingauthor}}
\address{Department of Mathematics, National Institute of Technology Meghalaya, Shillong 793003, India}
\cortext[mycorrespondingauthor]{Corresponding author}
\ead{manideepa.saha@nitm.ac.in}


%
%
%
%

\begin{abstract}Given a proper cone $K$ in the Euclidean space $\mathbb{R}^n$, a square matrix $A$ is said to be $K$-semipositive if there exists an $x\in K$ such that $Ax\in \text{int}(K)$, the topological interior of $K$. The paper aims to study algebraic and geometrical properties of $K$-semipositive matrices with  special emphasis on the self-dual proper Lorentz cone $\mathcal{L}^n_+=\{x\in \mathbb{R}^n:x_n\geq 0,\sum\limits_{i=1}^{n-1}x_{i}^2\leq x_n^2\}$. More specifically, we discuss a few necessary and  other sufficient algebraic conditions for $\mathcal{L}^n_+$-semipositive matrices. Also, we provide algebraic characterizations for diagonal and orthogonal  $\mathcal{L}^n_+$-semipositive matrices. Furthermore, given a square matrix $A$ and a proper cone $K$, geometric properties of the semipositive cone $\mathcal{K}_{A,K}=\{x\in K:~Ax\in K\}$ and the cone of $\mathcal{S}_{A,K}=\{x:Ax\in K\}$ are discussed in terms of their extremals. As $\mathcal{L}^n_+$ is an ellipsoidal cone, at last we find results for the  cones $\mathcal{K}_{A,\mathcal{L}^n_+}$ and $\mathcal{S}_{A,\mathcal{L}^n_+}$ to be ellipsoidal.
\end{abstract}

\end{frontmatter}


\section{Introduction}
All through the paper, we denote $\mathbb{R}^{m\times n}$ as the set of all matrices of order $m\times n$ with real entries, and the inequality signs, used for vectors/matrices, are meant to represent entry wise inequalities.   A \emph{cone} $K\subseteq\mathbb{R}^n$ has the properties that $K+K \subseteq K$ and $\alpha K \subseteq K$, for all $\alpha \geq 0$. We write $\Int(K) $ to represent the topological interior of the cone $K$, in the Euclidean space.  It is known that a set in $\mathbb{R}^n$ is called \emph{convex}, if it contains the line segment joining any two of its points. A \emph{proper cone} $K$ is  (convex) closed, pointed ($K\cap(-K)=\{0\}$), and solid (non-empty interior) cone. A proper cone $K$ in $\mathbb{R}^n$ always generates a \emph{ partial order}, defined as
\[x\overset{K}{\leq } y \text{ if and only if }y-x\in K\]
The dual of a cone $K\in \mathbb{R}^{n}$ is denoted by $K^{*}$ and  is defined as
\[K^{*}=\{x\in \mathbb{R}^{n}:\langle x,y\rangle \geq 0~\text{for all}~y\in K\}\]
where $\langle .,.\rangle$ denotes the standard inner product in the Euclidean space.

Given a closed convex cone $K$ in $\mathbb{R}^n$, a vector $x\in\mathbb{R}^n$ is called  an \emph{extremal} of $K$ if $0\overset{K}{\leq} y\overset{K}{\leq} x$ implies that $y=\alpha x$, for some $\alpha\geq 0$.  If a proper cone $K$ in $\mathbb{R}^n$, has  finitely many extremals, then it is known as a \emph{polyhedral cone} and,  a polyhedral cone with exactly $n$ extremals is called  as a {\em simplicial cone}. A matrix $A\in \mathbb{R}^{n\times n}$ is said to be \emph{$K$-monotone} if $Ax\in K$ implies $x\in K$.  In particular, if we choose $K=\mathbb{R}^{n}_{+}$ (the nonnegative orthant of $\mathbb{R}^{n}$), then $\mathbb{R}^{+}_{n}$-monotone matrices are the well-known monotone matrices~\cite{BerP94}.  Monotonicity and inverse positiveness are identical for square matrices~\cite{Col52}.  This interesting characterization of monotone matrices has generalized for $K$-monotone matrices, which is given in the following result and is used in our result in later section.
\begin{theorem}{\rm \cite{BerP94}}\label{monotone}
A matrix $A$ is $K$-monotone if and only if $A$ is non-singular and $A^{-1}\in \pi(K)=\{A\in\mathbb{R}^{n\times n}:~AK\subseteq K\}$,  the collection of $n\times n$ matrices that leave $K$ invariant..
\end{theorem}

A matrix $A\in\mathbb{R}^{m\times n}$ is said to be {\em semipositive} if there is an $x>(\geq)0$ such that $Ax>0$, and any such vector $x$ is known as a {\em semipositivity vector} of the matrix $A$. Importance of semipositive matrices are found in various problems like linear complementarity problems, characterizing invertible $M$-matrices, stability of matrices, game theory, optimization problems etc. Semipositive matrices were initially studied by  Fielder and Pt\'ak in \cite{FieP66} and several interesting and important properties that include algebraic, geometric, and spectral properties of semipositive matrices, can be found in \cite{ChoKS18,DorGJT16,HiaS20,JohKS94,TsaS17,Tsa15,Wer94}. One of the important subclass of semipositive matrices is minimally semipositive matrices whose none of the column deleted submatrix is semipositive.  In \cite{JohKS94}, it is proved that  semipositive matrix with nonnegative left inverse, is minimally semipositive and its converse viz., a minimally semipositive  matrix has a nonnegative left inverse. For square matrices, the statement is equivalent to say that a square matrix is minimally semipositive if and only if it is inverse positive. A semipositive matrix $A$ is associated with the proper cone $K_A=\{x\geq 0: Ax\geq 0\}$, which is known as {\em semipositive cone}. In \cite{TsaS17,Tsa15} it is established that a semipositive cone is a proper polyhedral cone, and further self-duality of a semipositive cone was also discussed. A brief discussion of the extremals of a semipositive cone can be found in \cite{HiaS20}. The relation between semipositive matrices and other different classes of matrices associated with the linear complementarity problem, is analyzed in \cite{ChoKS18}.

Given two proper cones $K_1$, $K_2$, respectively in $\mathbb{R}^n$ and $\mathbb{R}^m$, the positivity and semipositivity properties of a matrix $A\in \mathbb{R}^{m\times n}$ associated with these cones, fascinated many researcher since $1970$, and this notions are defined as follows:
\begin{definition} A matrix $A\in\mathbb{R}^{n}$ s said to be
\begin{itemize}
\item[\em (i)] $(K_1,K_2)$-\emph{negative} if $A(K_1)\subseteq K_2$. The set of all $(K_1,K_2)$-\emph{negative} matrices is denoted $\pi(K_1,K_2)$. 
\item[\em (ii)]$(K_1,K_2)$-\emph{positive} if $A(K_1\setminus\{0\})\subseteq \Int(K_2)$.
\item[\em (iii)]$(K_1,K_2)$-\emph{semipositive} if if there exists $x\in K_1$, such that $Ax\in \Int(K_2)$, which is equivalent to say that $A$ is $(K_1,K_2)$-semipositive  if  there is an $x\in \Int(K_1)$, such that $Ax\in \Int(K_2)$. Such a vector $x$ is called a semipositivity vector of $A$ with respect to the cones $K_1$ and $K_2$. Throughout the paper, we simply write $x$ is a semipositivity vector of $A$, if cones are clear from the context. The set of all $(K_1,K_2)$-\emph{semipositive} matrices is denoted $S(K_1,K_2)$. 
\end{itemize}
\end{definition} 
 If $K_1=K_2=K$ in the above definition, we write simply $A$ is \emph{$K$-negative}, or,  \emph{$K$-positive} or  \emph{$K$-semipositive}, respectively, and we use $\pi(K)$ and $S(K)$ for $\pi(K_1,K_2)$ and $S(K_1,K_2)$, respectively. It is worthwhile to mention that the semipositive matrix, defined in the above paragraph, is the $\mathbb{R}^n_{+}$-semipositive matrix. In ~\cite{SchV70}, it is shown that $\pi(K_1,K_2)$ is a proper cone.  Also, one can observe that $\Int(\pi(K_1,K_2))$ is the collection of all $(K_1,K_2)$-positive  matrices. We refer to ~\cite{ChaJM18,SchV70,Tam92,Van68}, and the references therein, for the various properties of the cone $\pi(K_1,K_2)$.
 
We now state a well-known result associated with $(K_1,K_2)$-semipositive matrices, which is a generalization of the Theorem of Alternative~\cite{Man69}.
\begin{theorem}{\em \cite{BerP94}} \label{one_is_true}
Let $K_1$ and $K_2$ be two proper cones in $\mathbb{R}^n$ and $\mathbb{R}^m$, respectively, and let $A\in \mathbb{R}^{m\times n}$, then exactly one of the following is true.
\begin{itemize}
\item[\rm(i)]  $A$ is $(K_1,K_2)$-semipositive.
\item[\rm(ii)] There exists an $x(\neq 0)\in -K_2^*$ such that $A^Tx\in K_1^*$.
\end{itemize}
\end{theorem}
 In 2018, the class of $(K_1,K_2)$-semipositive matrices are revisited in~\cite{ChaJM18} for finite dimensional real Hilbert spaces, and furnished many interesting results that include characterization of $K$-semipositive linear operators (matrices). We now state the following result due to~\cite{ChaJM18} that of our particular interest for further study.
 \begin{theorem}{\em \cite{ChaJM18}}\label{thm2-int} Let $K_1$ and $K_2$ be two proper cones in $\mathbb{R}^n$ and $\mathbb{R}^m$, respectively, and let $A\in\mathbb{R}^{m\times n}$. Then $A\in S(K_1,K_2)$ if and only if there exists an invertible matrix $\widetilde{X}\in\mathbb{R}^{n\times n}$ that is $K_1$-positive and $\widetilde{Y}\in\mathbb{R}^{m\times n}$  that is $(K_1,K_2)$-positive with $A=\widetilde{Y}\widetilde{X}^{-1}$.
\end{theorem}
 The above theorem generalizes Theorem 3.1 of~\cite{Tsa15}, which states that any semipositive matrix $A$ can be factored as $YX^{-1}$, for some positive matrices $X$ and $Y$.

Ellipsoidal cones are another important class of convex cones which have simple structure with long mathematical history. These cones are involved in applications of many fields like,  physics, statistics, optimization, etc.  An ellipsoidal cone is a proper cone that has at least one ellipsoidal cross section (intersection with hyperplane). Various definitions of ellipsoidal cones are available in literature and, a review and connection of these different approaches can be found in~\cite{SeeT20}. An \emph{ellipsoid} in $\mathbb{R}^n$ is an image of a closed unit ball in $\mathbb{R}^n$ under an invertible affine transformation on  $\mathbb{R}^n$.  We now provide a mathematical definition of an ellipsoidal cone, which is considered by Stern and Wolkowicz  in~\cite{SteW91a}.
\begin{definition}{\rm \cite{JerM13,SteW91a,SeeT20}}
A proper cone $K$ in $\mathbb{R}^n$ is an \emph{ellipsoidal cone} if there exists a nonzero vector $y_0\in\mathbb{R}^n$ such that  the cross section $S(K,y_0):=\{x\in K:~x^Ty=1\}$ is an ellipsoid in the hyperplane $\{x\in\mathbb{R}^n:~ x^Ty_0=1\}$.
\end{definition}

Stern and Wolkowicz \cite{SteW91a} furnish a characterization of ellipsoidal cones,  which is given by the following results.

\begin{theorem} {\rm \cite{SteW91a}}  A proper cone $K$  is an ellipsoidal cone if and only if $K$ admits the representation
\[K=\{x\in \mathbb{R}^n:x^TQx\leq 0:u^Tx\geq 0\}\]
where $Q$ is a symmetric non-singular  matrix with inertia $(n-1,0,1)$ and $u$ is an eigenvector of $Q$ associated to the unique negative eigenvalue  $\lambda$.
\end{theorem}
 Taking the non-singular symmetric matrix $Q=\diag(1,\ldots,1,-1)$ and considering the only negative eigenvalue $\lambda=-1$ with the corresponding eigenvector $u=[0,0,\ldots, 1]^T$ in the above theorem, the associated ellipsoidal cone takes the form 
\[K=\left\{x\in \mathbb{R}^n:x_n\geq 0,~\sum\limits_{i=1}^{n-1}x_{i}^2\leq x_n^2\right\}\]
This  particular ellipsoidal cone is known as \emph{Lorentz cone} and is denoted by $\mathcal{L}^n_+$.  In fact any ellipsoidal cone can be characterized by the self dual Lorentz cone, which are given below:

\begin{theorem}{\rm\cite{JerM13}} For a proper cone $K$, following statements are equivalent:
\begin{itemize}
\item[\rm(a)] $K$ is an ellipsoidal cone.
\item[\rm(b)] $K$ is isomorphic to $\mathcal{L}^n_+$.
\item[\rm(c)] The dual cone $K^{*}$ is an ellipsoidal cone.
\end{itemize}
\end{theorem}
\begin{theorem}{\rm \cite{SteW91}}\label{ellip}
A proper cone $K$ is ellipsoidal if and only if $K=X(\mathcal{L}^n_+)$ for some invertible matrix $X$.
\end{theorem}

In literature ``Lorentz cones are also known as ice cream cones or second-order cones''.  Nevertheless Lorentz cone is the only self-dual revolution cone upto orthogonal transformation. This type of cones appears in second order cone programming (SOCP) problems.  Applications of SOCP problems may be found in various engineering problems, like truss design, filter design, antenna array weight design, and grasping forte optimization in robotics etc.~\cite{LobVBL98}.  In \cite{ChaJM18}, Chandrashekaran et al. characterized $\mathcal{L}^n_+$-semipositive matrices in terms of semipositive matrices. In particular, they proved that 
each $\mathcal{L}^n_+$-semipositive matrix is similar to a semipositive matrix via an $\left(\mathbb{R}^n_+,\mathcal{L}^n_+\right)$-negative matrix. The formal statements of the results are given below.

\begin{theorem}{\em\cite{ChaJM18}}\label{thm3-int} If $A\in\mathbb{R}^{2\times 2}$ is a semipositive matrix, then $TAT^{-1}$ is $\mathcal{L}^2_+$-semipositive, where $T=\left[\begin{array}{lr} 1 &-1\\1 &1\end{array} \right]$. Conversely, if $C$ is $\mathcal{L}^2_+$-semipositive, then there exists a semipositive matrix $A$ such that $C=TAT^{-1}$. 
\end{theorem}

\begin{theorem}{\em\cite{ChaJM18}}\label{thm4-int} Let $n\geq 3$. If $A$ is a semipositive matrix, then $SAS^{-1}$ is $\mathcal{L}^n_+$-semipositive, for some invertible matrix $S\in\pi\left(\mathbb{R}^n_+,\mathcal{L}^n_+\right)$, Conversely, if $C$ is $\mathcal{L}^n_+$-semipositive, then there exists a semipositive matrix $A$ such that $C=\left(S^T\right)^{-1}AS^T$. 
\end{theorem}

As similarity transforms preserve many algebraic and spectral properties, in order to study those properties for  $\mathcal{L}^n_+$-semipositive matrices, it suffices to study the same for semipositive matrices. In this article, we provide some results pertinent to $K$-semipositive and $\mathcal{L}^n_+$-semipositive matrices, which may not be preserved under similarity transformation.  The aim of the paper is to study  $K$-semipositive matrices for various types of proper cone $K$. An emphasis is given to analyze the properties of  $\mathcal{L}^n_+$-semipositive matrices, and as  particular cases, we study orthogonal and triangular $\mathcal{L}^n_+$-semipositive matrices.

The outline of the paper is as follows: Algebraic properties of $\mathcal{L}^n_+$-semipositive matrices are discussed in Section~\ref{sec2}. In particular, we provide a few necessary and some other sufficient conditions for $\mathcal{L}^n_+$-semipositive matrices. In Section~\ref{sec3}, an emphasize has been given to special types of $\mathcal{L}^n_+$-semipositive matrices.  More precisely, characterization of diagonal, orthogonal and lower triangular $\mathcal{L}^n_+$-semipositive matrices are studied. Later in Section~\ref{sec4}, given a square matrix $A$ and a  proper cone $K$, geometrical properties of the $K$-semipositive cone  $\mathcal{K}_{A,K}=\{x\in K:Ax\in K\}$ and  of the cone $\mathcal{S}_{A,K}=\{x:Ax\in K\}$, are discussed. Especially, an emphasize has been given to the cone $\mathcal{S}_{A,\mathcal{L}^n_+}$. Lastly, we end the paper with conclusive remarks in Section~\ref{sec5}.

\section{Algebraic properties of $\mathcal{L}^n_+$-semipositive matrices}\label{sec2}
As mentioned in the previous section, semipositivity of matrices catches interest of researchers mainly due to its relation with solvability of LCPs. {Given a proper cone $K$, Theorem~\ref{thm2-int} provides a characterization of a $K$-semipositive matrix $A$ by decomposing  $A=\widetilde{Y}\widetilde{X}^{-1}$, for some $K_1$-positive matrix $\widetilde{X}$ and $(K_1,K_2)$-positive $\widetilde{Y}$~\cite{ChaJM18}. This decomposition was exhibited for $\mathcal{L}^n_{+}$-semipositive matrix in  ~\cite{ChaJM18}, which is stated below:
\begin{theorem}{\em\cite{ChaJM18}}\label{th1-sec1}
$A\in\mathbb{R}^{n\times n}$ if $\mathcal{L}^n_+$-semipositive if and only if there exists $X,Y\in\mathbb{R}^{n\times n}$ with $X$ invertible, $X,Y$ are $\mathcal{L}^n_{+}$-semipositive such that $A=YX^{-1}$.
\end{theorem}

 Motivated by this work, we revisit the class of $\mathcal{L}^n_+$-semipositive matrices and presents some algebraic properties of such matrices.
}

Recall that $\mathcal{L}^{n}_{+}=\left\{x\in \mathbb{R}^n:x_n\geq 0,~\sum\limits_{i=1}^{n-1}x_{i}^2\leq x_n^2\right\}=\left\{x:~\norm{x}\leq \sqrt{2}x_n\right\}$, where $\norm{.}$ represents the Euclidean norm in $\mathbb{R}^{n}$, which together with Cauchy-Schwarz inequality lead to following result:
\begin{lemma}\label{lem1} For $x,y,z\in\mathcal{L}_{n}^{+}$ we have
\begin{itemize}
\item[\rm(a)] $\sum\limits_{k=1}^{n-1}(x_{k}y_{k})\leq x_{n}y_{n}$. The equality holds if $x,y\in\partial \mathcal{L}_{n}^{+}$ and are colinear.
\item[\rm(b)] $\sum\limits_{k=1}^{n-1}(x_{k}y_{k}z_{k})\leq x_{n}y_{n}z_{n}$. The equality holds if $x,y,z\in\partial \mathcal{L}_{n}^{+}$ and are colinear.
\item[\rm(c)]  For any positive integer $l$, $\mathcal{L}_{n}^{+}$ contains the vector $x^{(l)}$, obtained by raising $l$-th power of entries of $x$, that is, $x^{(l)}=\left(x_{1}^{l},x_{2}^{l},\ldots,x_{n}^{l}  \right)\in\mathcal{L}_{n}^{+}$.
\end{itemize}
\end{lemma}

Next result on $\mathcal{L}^n_+$-semipositive matrices, is an immediate consequence of the definition and hence the proof is skipped.

\begin{proposition}\label{Prop1}
For $A=[\widetilde{a}_1,\ldots,\widetilde{a}_{n}]\in \mathbb{R}^{n\times n}$, the followings results hold:
\begin{itemize}
\item [\rm(i)] If $A$ is $\mathcal{L}^n_+$-semipositive, then  so is $\alpha A$, for any $\alpha>0$.
\item [\rm(ii)] If $\widetilde{a}_{n}\in\Int(\mathcal{L}^n_+)$, then $A$ is $\mathcal{L}^n_+$-semipositive with $e_n$ as the semipositivity vector.
\item [\rm(iii)] For any permutation matrix $P=(p_{ij})$  with $p_{nn}=1$,  $A$ is $\mathcal{L}^n_+$-semipositive if and only if $PA$ is $\mathcal{L}^n_+$-semipositive.
\item [\rm(iv)] If $A$ is block lower triangular matrix of the form 
\[A=\left[\begin{array}{lc}
A_{11} & 0  \\
A_{21} & A_{22}
\end{array}\right]\]
where $A_{22}\in\mathbb{R}^{k\times k}$, then $A_{22}$ is $\mathcal{L}^k_+$-semipositive implies $A$ is $\mathcal{L}^n_+$-semipositive. In that case if $x$ is a semipositivity vector of $A_{22}$, then $z=[0~~ x]^T$ is a semipositivity vector of $A$.
\end{itemize}
\end{proposition}
\begin{rk}\label{rk1}\rm From the above result we observe that the class $S\left(\mathcal{L}^n_+ \right)$ of $\mathcal{L}^n_+$-semipositives is closed with respect to nonnegative scalar multiplication, but following example shows that $S\left(\mathcal{L}^n_+ \right)$ is not a cone in $\mathbb{R}^{n\times n}$.
\end{rk}

\begin{example}\rm 
The matrices $A=\left[\begin{array}{ll}
1 & 4 \\
5 & 3 
\end{array}\right]$ and $B=\left[\begin{array}{ll}
4 & 3 \\
1 & 2 
\end{array}\right]$ are $\mathcal{L}^2_+$-semipositive  matrices with semipositivity vectors $[\frac{1}{2},1]^T$ and $[-\frac{1}{2},1]^T$, respectively. We show that $A+B$ is not $\mathcal{L}^2_+$-semipositive. Suppose that there exists a vector $x\in \mathcal{L}^2_+$ such that 
\[(A+B)x=\left[\begin{array}{l}
5x_1+7x_2 \\
6x_1+5x_2
\end{array}\right]\in \Int(\mathcal{L}^2_+)\]
Then $|x_1|\leq x_2$ and $|(5x_1+7x_2)|<6x_1+5x_2$, so that $2x_2<x_1$. Since $x_2\geq 0$ and hence $x_1>0$,  $2x_2<x_1$ contradicts the inequality  $x_1\leq x_2$.

\end{example}

Next theorem presents a characterization of rank one $\mathcal{L}^n_+$-semipositive matrices. In proving the necessary part of the following result, a similar argument is used as that of available for Case 1 in Theorem 2.1 of~\cite{ChaJM18}.

%

\begin{theorem}
Let $A=uv^T\in \mathbb{R}^{n\times n }$ be a non-zero rank one matrix with $u_n\geq 0$. Then $A$ is an $\mathcal{L}^n_+$-semipositive matrix if and only if $u\in \Int (\mathcal{L}^n_+)$ and $v\notin -\mathcal{L}^n_+$.
\end{theorem}
\begin{proof}
Let $A$ be an $\mathcal{L}^n_+$-semipositive matrix. Choose an $x\in \Int (\mathcal{L}^n_+)$ such that $Ax\in \Int(\mathcal{L}^n_+)$. Then $uv^Tx\in \Int(\mathcal{L}^n_+)$, and hence  $v^Tx\neq 0$.  If $v^Tx<0$, we get $u\in -\Int(\mathcal{L}^n_+)$, which is not possible as $u_n\geq 0$. Therefore $v^Tx>0$, which implies that $u\in \Int(\mathcal{L}^n_+)$ and $v\notin -\mathcal{L}^n_+$.


Conversely, assume that $v\notin -\mathcal{L}^n_+$ and $u\in \Int(\mathcal{L}^n_+)$. If $v\in \mathcal{L}^n_+$, then any $x\in \Int (\mathcal{L}^n_+)$ is a semipositivity vector.  Suppose that $v\notin \mathcal{L}^n_+ \cup -\mathcal{L}^n_+$. We now have the following two cases.\\[2mm]
\textbf{Case I}: Let $v_n\geq  0$. We show that  $q=v+\norm{v}e_n\in \mathcal{L}^n_+$ is a semipositivity vector of $A$. Notice that $q_n=v_n+\norm{v}> 0$, since $v\neq 0$ and
\[\sum_{i=1}^{n-1}q^2_i=\sum_{i=1}^{n-1}v^2_i\leq \norm{v}^2\leq (v_n+\norm{v})^2=q_n^2\]
Thus $q\in \mathcal{L}^n_+$. Now $\langle q,v \rangle =\norm{v}^2+v_n> 0$, and so
$Aq=uv^Tq\in \Int(\mathcal{L}^n_+)$, as $u\in \Int(\mathcal{L}^n_+)$.\\[2mm]
\textbf{Case II}: Assume that $v_n<0$. We claim that the vector $p=(\norm{v}^2-\epsilon)e_n-v_nv$ is a semipositivity vector of $A$, for suitably chosen $\epsilon>0$. Observe that 
\[p_n=\norm{v}^2-v_n^2-\epsilon,~~\text{ and }~~\sum_{i=1}^{n-1}p^2_i=v_n^2\sum_{i=1}^{n-1}v^2_i=v_n^2(\norm{v}^2-v_n^2)\]
As $v\notin \mathcal{L}^n_+ \cup -\mathcal{L}^n_+$, so $v_n^2<\norm{v}^2-v_n^2$ and hence  \[\sum_{i=1}^{n-1}p^2_i<(\norm{v}^2-v_n^2)^2\]

Choose $\epsilon>0$ small enough so that  
\[p_n=\norm{v}^2-v_n^2-\epsilon>0~~\text{ and }~~\sum_{i=1}^{n-1}p^2_i<(\norm{v}^2-v_n^2-\epsilon)^2=p_n^2\]
Then $p\in \Int (\mathcal{L}^n_+)$. Since $v_n<0$, so we have that
\[v^Tp=\langle p,v \rangle = \norm{v}^2v_n-v_n\norm{v}^2-\epsilon v_n=-\epsilon v_n>0\]
Since $u\in \Int(\mathcal{L}^n_+)$, so we get $Ap=uv^Tp\in \Int(\mathcal{L}^n_+)$. Thus the result holds.
\end{proof}
%

In Proposition~\ref{Prop1}(ii), we observe that if the end column of a matrix lies in  $\Int(\mathcal{L}^n_+)$, then the matrix is $\mathcal{L}^n_+$-semipositive with semipositivity vector  $e_n $. In fact, if the end column lies in the boundary $\partial\mathcal{L}_{n}^{+}$, then also the matrix is $\mathcal{L}_{n}^{+}$-semipositive under certain condition. The details of it follows in the next theorem.
\begin{theorem}\label{thm1}
Let $A\in \mathbb{R}^{n\times n}$ be a matrix such that  $\mathcal{L}^n_+$ contains the $n$-th column.  For some $k\neq n$, if the $k$-th column of $A$ belongs to $\Int\left(\mathcal{L}^n_+\right)$, then $A$ is $\mathcal{L}^n_+$-semipositive.
\end{theorem}
\begin{proof}
Write $A=(a_{ij})=[\widetilde{a}_1,\ldots,\widetilde{a}_n]$. By the given hypothesis we have  that
\begin{equation}\label{eq1}
\norm{\widetilde{a}_{k}}\leq \sqrt{2}a_{nk},~~\text{ and } \norm{\widetilde{a}_{n}}\leq \sqrt{2} a_{nn}
\end{equation}
Take $0< \alpha<1$, and define a vector $x\in\mathbb{R}^n$ by
\[x_{i}=\left\{\begin{array}{ll}
\alpha &\text{ if } i=k\\
1 &\text{ if } i=n\\
0 &\text{ otherwise}
\end{array} \right.\]
so that $Ax=[a_{1k}\alpha +a_{1n},a_{2k}\alpha+a_{2n},\ldots, a_{nk}\alpha+a_{nn} ]^T=\alpha \widetilde{a}_{k}+\widetilde{a}_{n}$.   Note that 
$(Ax)_n=a_{nk}\alpha+a_{nn}\geq 0$  and~\eqref{eq1} implies that
\begin{align}\label{eqth2.2}
\norm{Ax}=\norm{\alpha \widetilde{a}_{k}+\widetilde{a}_{n}}\leq \alpha\norm{\widetilde{a}_{k}}+\norm{\widetilde{a}_{n}}< \sqrt{2}(a_{nk}\alpha +a_{nn})=\sqrt{2}(Ax)_n
\end{align}
Hence the result follows.
\end{proof}
%

The conditions in Theorem~\ref{thm1} are not necessary for the semipositivity of matrices over $\mathcal{L}_{+}^{n}$.  Following examples illustrate this fact.
\begin{example}\rm
The matrix $A=\left[\begin{array}{ll}
1 & 4 \\
5 & 3 
\end{array}\right]$is an $\mathcal{L}^2_+$-semipositive matrix with semipositivity vector $y=[\frac{1}{2},1]^T$, and the end column doesn't belong to $\mathcal{L}^2_+$.
\end{example}
In fact, next example provides an  $\mathcal{L}^n_+$-semipositive matrix with no column belongs to $\mathcal{L}^n_+$.
\begin{example}\rm
The matrix $A=\left[\begin{array}{ll}
4 & 3 \\
1 & 2 
\end{array}\right]$ is an $\mathcal{L}^2_+$-semipositive matrix with semipositivity vector $y=[-\frac{1}{2},1]^T$, whereas none of the column belongs to $\mathcal{L}^{n}_{+}$.
\end{example}

In Remark~\ref{rk1}, we notice that sum of two $\mathcal{L}^n_+$-semipositive matrices is not necessarily an $\mathcal{L}^n_+$-semipositive matrix. In the succeeding  theorem, we present a weaker class of matrices for which this holds.
\begin{theorem}\label{thm2}
Suppose that  $D=\diag(d_1,\ldots,d_n)$ is a diagonal matrix with $\norm{D}=d_n$.  If  $A$ is an $\mathcal{L}^n_+$-semipositive matrix, then so is $A+D$.
\end{theorem}
\begin{proof}
Since $A$ is $\mathcal{L}^n_+$-semipositive, choose a vector $x\in\Int(\mathcal{L}^n_+)$ such that $y=Ax\in \mathcal{L}^n_+$. Note that  $(y_n+d_nx_n)\geq 0$ and as shown in equation~\eqref{eqth2.2},
 $A+D\in S(\mathcal{L}_{+}^{n})$ with the same semipositivity vector.
\end{proof}

 \begin{rk}\rm  Theorem~\ref{thm2} holds for all diagonal matrix $D=\diag(d_1,d_2,\ldots,d_n)$ with $d=[d_1,\ldots,d_n]^T\in  \mathcal{L}^n_+$.
\end{rk}

Theorem~\ref{thm1} gives a sufficient condition for $\mathcal{L}^n_+$-semipositive matrices. We now furnish a necessary condition for the same.
\begin{theorem}\label{thm3}
If $A$ is a  $\mathcal{L}^n_+$-semipositive matrix, then  $n$-th row vector of $A$ does not belongs to $-\mathcal{L}^n_+$.
\end{theorem}
\begin{proof}
As $A$ be a $\mathcal{L}^n_+$-semipositive, there exist $x\in\Int(\mathcal{L}^n_+)$ such that  $\norm{Ax}<\sqrt{2}(Ax)_n$ and $(Ax)_n>0$, that is, $\langle a_n, x\rangle >0$, where $a_n$ is the $n$-th row vector of $A$. If $a_n\in -\mathcal{L}^n_+$, then $-a_n,x\in \mathcal{L}^n_+$ implies that $\langle -a_n, x\rangle \geq 0$,  which is a contradiction.
\end{proof}

Succeeding example explains the fact that the converse of Theorem~\ref{thm3} may not be true in general.
\begin{example}\rm
The 3rd row vector of the matrix $A=\begin{bmatrix}
1 & 1 & 2 \\
1 & 1 & 4 \\
1 & 1 & 1
\end{bmatrix}$ does not belongs to $-\mathcal{L}_3^+$. Taking the vector $y=(1, 1, -3)^T$, it can be observed that $y\in -\mathcal{L}^n_+$ and $A^Ty\in \mathcal{L}^n_+$, therefore by Theorem~\ref{one_is_true}, $A$ is not $\mathcal{L}^n_+$-semipositive.
\end{example} 
The previous example exhibits the fact that the condition provided in Theorem~\ref{thm3}, is not sufficient to check the semipositivity of matrices over $\mathcal{L}^{n}_{+}$. We end this section by giving a sufficient condition for the same.
\begin{theorem}\label{thm4} Let $A=[a^T_1,\ldots,a^T_n]^T$ be a  matrix satisfying
\begin{equation}\label{eq1-thm4}
\norm{a_1}^2+\ldots+\norm{a_{n-1}}^2< \frac{1}{2}\norm{a_n}^2 \text{  and }~a_{nn}\geq 0.
\end{equation}
Then $A$ is $\mathcal{L}^n_+$-semipositive.
\end{theorem}
\begin{proof} We show that $A$ is $\mathcal{L}^n_+$-semipositive, by proving that  $y=\dfrac{a_n}{\norm{a_n}}+e_n$ is a semipositivity vector of $A$.

First we prove that $y\in \mathcal{L}^n_+$. Note that $y_n=\dfrac{a_{nn}}{\norm{a_n}}+1\geq 1>0$. We now have that 

\begin{align*}
\sum_{i=1}^{n-1} y_i^2& =\sum_{i=1}^{n-1}\frac{a_{ni}^2}{\norm{a_n}^2} =1-\frac{a_{nn}^2}{\norm{a_n}^2}\leq 1\leq y_n^2
 \end{align*} 
 Thus we have $y\in \mathcal{L}^n_+$.
 
 Next we show that $Ay\in \Int(\mathcal{L}^n_+)$. To do so, we calculate the terms $(Ay)_n$ and $(Ay)_1^2+\ldots+(Ay)_{n-1}^2 $.  Now,
  \[(Ay)_n=\langle a_n, y \rangle=\left\langle a_n, \dfrac{a_n}{\norm{a_n}}+e_n \right\rangle=(\norm{a_n}+a_{nn})\geq0\]
 Next, we have that
 \begin{align*}
\sum_{i=1}^{n-1}(Ay)_i^2= \sum_{i=1}^{n-1}  \langle a_i, y \rangle &\leq \sum_{i=1}^{n-1}  \norm{y}^2~ \norm{a_i}^2 \\
&\leq \norm{y}^2 \sum_{i=1}^{n-1}  \norm{a_i}^2\\
  & <2y_n^2\frac{1}{2}\norm{a_n}^2~~~~~~~~~~~[\text{ since }y\in\mathcal{L}^n_+ \text{ and }\eqref{eq1-thm4}]\\
  &=(\norm{a_n}+a_{nn})^2=(Ay)_n^2
 \end{align*}
 This shows that  $A$ is $\mathcal{L}^n_+$-semipositive.
\end{proof}

\section{ Special type of $\mathcal{L}^n_+$-semipositive Matrices}\label{sec3}
In the previous section we provided a few necessary conditions  and  some other sufficient conditions for a given matrix to be $\mathcal{L}^n_+$-semipositive matrix. But no conditions that characterize  $\mathcal{L}^n_+$-semipositive matrices were discussed. However a few characterization of $\mathcal{L}^n_+$-semipositive matrices can be found in~\cite{ChaJM18}. This section is mainly focused on characterization  of special classes of  $\mathcal{L}^n_+$-semipositive matrices. In particular, we study diagonal, orthogonal and lower triangular $\mathcal{L}^n_+$-semipositive matrices. 

We begin with the simplest subclass $\mathcal{L}^n_+$-semipositive matrices, that is, diagonal $\mathcal{L}^n_+$-semipositive matrices.  Observe that if the $(n,n)$-th entry of a diagonal matrix $D$ is positive, then $D$ is a semipositive matrix (with a semipositivity vector  $e_n$). In fact, the converse also follows from the definition and hence we have the following result.



%
\begin{theorem}
A diagonal matrix $D$ is $\mathcal{L}^n_+$-semipositive if and only if the $(n,n)$-th term is positive.
\end{theorem}


For further discussion we use the following notations:
\begin{itemize}
\item[(i)] $H_{+}:=\{x:x_n\geq 0\}$, the (closed) upper halfspace of $\mathbb{R}^n$.
\item[(ii)] $H_{-}:=\{x:x_n\leq 0\}$, the (closed) lower halfspace of $\mathbb{R}^n$.
\item[(iii)] $H_{0}:=\{x:x_n= 0\}$, the hyperplane in $\mathbb{R}^n$ that separates $H_{+}$ and $H_{-}$.
\item[(iv)] $\mathbb{S}^{n}=\{x:\norm{x}=1\}$, the unit circle in $\mathbb{R}^{n}$.
\item[(v)] $L(x,y)=\{tx+(1-t)y:~~0\leq t\leq 1\}$, the line joining the points $x$ and $y$ in $\mathbb{R}^n$.
\end{itemize}

Note that the hyperplane $H_0$ intersects $\mathcal{L}^n_+$ only in the origin and  angle between any two vectors of $\mathcal{L}^n_+$ is less than or equal to $\pi/2$. We next furnish a necessary and sufficient condition for orthogonal $\mathcal{L}^n_+$-semipositive matrix. To do so, we require the following two lemmas. 
\begin{lemma}\label{angle1}
Any two non-zero vectors, respectively, from the sets $H_0$, and $\partial \mathcal{L}^n_+$ creates a least angle of  $\pi/4$.
\end{lemma}
\begin{proof}
It is equivalent to find the minimum value of $\arccos \langle a, z \rangle , a\in H_0, z\in \partial  \mathcal{L}^n_+$. We assume that both $a,z$ are of unit norm for simplicity. Therefore, the solution of the desired minimization problem is same as the following maximization problem, which is described below:
\begin{center}
 \begin{equation*}
      \begin{array}{l}
       \max~~ \langle a,z \rangle,~~\text{ subject to } \\
       \left\{\begin{array}{l}
    z\in \partial \mathcal{L}^n_+\cap \mathbb{S}^n \\
      a\in H_0 \cap \mathbb{S}^n 
      \end{array} \right.
          \end{array}  
  \end{equation*}  
  \end{center}

Since $z\in \partial \mathcal{L}^n_+\cap \mathbb{S}^n $, so $z_n=\frac{1}{\sqrt{2}}$ and also $a\in H_0$ implies $a_n=0$. Thus the above maximization problem can be re-written in a simpler form given by
 \begin{equation}\label{lem2-eq1} 
      \begin{array}{ll}
      \max ~~a_1z_1+\ldots+a_{n-1}z_{n-1},~~\text{ subject to } \\
        \left\{   \begin{array}{l}z_1^2+\ldots+z_{n-1}^2=\frac{1}{2} \\
        a_1^2\ldots+a_{n-1}^2=1  \end{array} \right.
          \end{array} 
  \end{equation}  
Solving the maximization problem~\eqref{lem2-eq1} by Lagrange multiplier method,  the solution is obtained as $\dfrac{1}{\sqrt{2}}$. This show that the minimum value of $\arccos\langle a,z \rangle, a\in H_0,  z\in \partial \mathcal{L}^n_+$ is $\pi/4$. 
\end{proof}

In the previous lemma it can be shown that, $\arccos\langle a,z \rangle \leq \pi/4, ~\forall z\in \mathcal{L}^n_+$.

\begin{lemma}\label{angle2}
For any $x=(x_1,\ldots,x_n)\in \mathbb{S}^n$, there exists a vector $z \in \Int(\mathcal{L}^n_+)$ with $\arccos\langle x,z \rangle <\pi /4$ if and only if $x_n>0$.
\end{lemma}
\begin{proof}
Assume that $x_n>0$. We show that $z=\frac{1}{2}(x+e_n)$ is the desired vector. We first prove that $z\in \Int(\mathcal{L}^n_+)$. Note that $z_{n}=\frac{1}{2}(x_n+1)>\frac{1}{2}>0$. Now,
 \begin{align*}
\norm{z}^2 & =\dfrac{1}{4}\left[x_1^2+\ldots+x_{n-1}^2+(x_n+1)^2\right]\\
  & =\frac{1}{4}\left(\norm{x}^2+2x_n+1\right)\\
  &=\frac{1}{2}(1+x_n)=z_n<2z_{n}^2
 \end{align*}

Since $\langle x,z \rangle =\frac{1}{2}(1+x_n)=||z||^2$, and $x_n>0$, we have 
\[\cos \theta =\frac{\langle x,z \rangle}{||z||}=\frac{1}{\sqrt{2}}\sqrt{1+x_n}>\frac{1}{\sqrt{2}}\]
where $\theta=\arccos\langle x,z \rangle$.
  This shows that $\arccos \langle x,z \rangle<\pi/4$

Conversely, let there exist a $z\in \Int(\mathcal{L}^n_+)\cap \mathbb{S}^n$ with  $\arccos\langle x,z \rangle <\pi /4$. Then $\langle x,z \rangle >\dfrac{1}{\sqrt{2}}$. Suppose that $x_n\leq 0$. Then $x\in H_-$ and so we have $z\in \Int(\mathcal{L}^n_+)$ and $x\notin \mathcal{L}^n_+$. This implies there exists a vector $p\in \partial \mathcal{L}^n_+\cap L(x,z) $. Again $z\in \Int(H_+)$, and $x\notin \Int(H_+)$, imply that we can find another vector $q\in H_0 \cap L(x,z)$ (in case $x_n=0$, we take $q=x$).  Then by Lemma \ref{angle1}, $\arccos \langle p,q \rangle \geq \pi/4$ and hence $\arccos \langle x,z \rangle \geq \pi/4$, because $p,q\in L(x,z)$, which is a contradiction. 
 Hence $x_n>0$.
 \end{proof}

We now state  and prove the main result.

\begin{theorem}
 An orthogonal matrix $A=(a_{ij})$ is $\mathcal{L}^n_+$-semipositive if and only if $a_{nn}>0$.
\end{theorem}
\begin{proof}
Since $A$ is an orthogonal matrix, we may assume that any semipositivity vector $z$  of $A$ has $\norm{z}=1$ and must satisfy $z\in \Int(\mathcal{L}^n_+)$  and $1=\norm{Az}<\sqrt{2}(Az)_n$. 
 So finding a semipositivity vector of $A$ over $\mathcal{L}^n_+$ is equivalent to find an unit vector $z$ such that 
 \[\langle a_n,z \rangle >\frac{1}{\sqrt{2}},~~~\text{ where }a_{n}^T~\text{ is the }n\text{-th row of }A\]

If $\theta $ is the angle between $z$ and $a_n$, then  $\cos \theta >\dfrac{1}{\sqrt{2}}$ if and only if $\theta <\pi/4$.  Hence by Lemma \ref{angle2}, $z$ is a semipositivity vector of $A$ if and only if $a_{nn}>0$.
\end{proof}
%


We study our last class of $\mathcal{L}^n_+$-semipositivity matrices which are lower triangular, under consideration of the section.  For a lower triangular matrix $A$, $A(a_{nn}e_{n})=a_{nn}^2e_{n}$ and hence the following results are obvious.

\begin{proposition}\label{prop1}
Let $A=(a_{ij})\in \mathbb{R}^{n\times n  }$ be a lower triangular matrix with $a_{nn}> 0$, then $A $ is $\mathcal{L}^n_+$-Semipositive.
\end{proposition}
\begin{proposition}
If $A$ be a lower triangular matrix with $a_{nn}\neq 0$, then there exists $x\in \mathcal{L}^n_+\cup (-\mathcal{L}^n_+)$ such that $Ax \in \mathcal{L}^n_+$.
\end{proposition}
%

We end this section by furnishing two sufficient conditions for semipositivity of lower triangular matrices over $\mathcal{L}^{n}_{+}$.
\begin{theorem}
Let $A=(a_{ij})=[a^T_1,\ldots,a^T_{n}]^T\in \mathbb{R}^{n\times n }$ be a non-zero lower triangular matrix such that $a_n\notin -\mathcal{L}^n_+$.  Let $\beta=-a_{nn},~~\gamma=\sum\limits_{i=1}^{n-1}\norm{a_i}^2$ and $\alpha_k =\sqrt{a_{1k}^2+\ldots+a_{(n-1)k}^2}$ for $k=1,\ldots,n$. Then we have the followings:
\begin{itemize}
\item [\rm(i)] If $\alpha_k+\beta <a_{nk} $, for $k=1,\ldots,n$,  then $A$ is $\mathcal{L}^n_+$-semipositive.
\item [\rm(ii)] If $\gamma+2\beta < \norm{a_n}$,
then $A$ is $\mathcal{L}^n_+$-semipositive.
\end{itemize}
\end{theorem}
\begin{proof} If $a_{nn}>0$, then the conclusion is obvious from Proposition~\ref{Prop1}. So, we assume that $a_{nn}\leq 0$.\\[2mm]
(i)  Let $\alpha_k+\beta <a_{nk}$.  We can always find a real number $c\in(0,1)$ such that 
 $\alpha_k+\beta <c a_{nk}$. Define $x\in\mathbb{R}^n$ by
 \[x_i=\left\{\begin{array}{ll}
 c &\text{ if }i=k\\
 1 &\text{ if }i=n\\
 0 &\text{ otherwise }\\
\end{array}  \right.\]
Then $x\in \Int\left(\mathcal{L}^n_+\right)$. We show that $x$ is the required semipositivity vector of $A$ over $\mathcal{L}^n_+$. Observe that $(Ax)_n=a_{nk}x_k+a_{nn}=a_{nk}c-\beta>\alpha_{k} \geq 0$. Furthermore, 
$$(Ax)_i=\langle a_i,x\rangle=\left\{
      \begin{array}{ll}
       0  &\text{if}~~ 1\leq i<k \\
        a_{ik}c &\text{if} ~~k\leq i \leq n-1 \\
        a_{nk}c+a_{nn}   &\text{if} ~~ i=n
          \end{array}
      \right.$$

      Thus we have
\[\sum_{i=1}^{n-1}(Ax)_i^2=\sum_{i=k}^{n-1}a_{ik}^2c^2=c^2\alpha_k^2<\alpha_k^2<(a_{nk}c-\beta)^2=(Ax)_n^2\]
Hence $A$ is $\mathcal{L}^{n}_{+}$-semipositive.\\

\noindent
(ii) Assume that $\gamma+2\beta \leq \norm{a_n}$. We prove that  $y=\dfrac{1}{3\norm{a_n}}a_n+\dfrac{2}{3}e_n$ is a semipositivity vector of $A$.

We first show that $y\in \mathcal{L}^n_+$. Notice that 
$y_n=\dfrac{1}{3}\left(\dfrac{a_{nn}}{\norm{a_n}}+2\right)>\dfrac{1}{3}>0$. Again,
\begin{align*}
 \sum_{i=1}^{n-1} y_i^2 & =\frac{1}{9{\norm{a_n}^2}}\sum\limits_{i=1}^{n-1}a_{ni}^2 \\
  & =\frac{1}{9}\left(1-\frac{a_{nn}^2}{\norm{a_n}^2}\right)<\frac{1}{9} <y_n^2
 \end{align*}
It remains to show that $Ay\in \Int(\mathcal{L}^n_+)$. Then by the hypothesis we have,
\[(Ay)_n=\left\langle a_n, \frac{1}{3\norm{a_n}}a_n+\frac{2}{3}e_n\right\rangle=\frac{1}{3}\norm{a_nn}+\frac{2}{3}a_{n}=\dfrac{1}{3}(\norm{a_{n}}-2\beta)> \dfrac{\gamma}{3}\geq 0 \]
and 
\begin{align*}
  \sum_{i=1}^{n-1}(Ay)_i^2 & =\dfrac{1}{9\norm{a_n}^2}\left(\sum_{i=1}^{n-1}\langle a_i,a_n \rangle ^2 \right)\\
  & \leq \frac{1}{9}\sum_{i=1}^{n-1}\norm{a_i}^2\\
  &<\frac{1}{9}(\norm{a_n}-2\beta)^2 =(Ay)_n^2
 \end{align*}
 This shows that $A$ is $\mathcal{L}^{n}_{+}$-semipositive.
\end{proof}

\section{Geometric properties of $\mathcal{L}^n_+$-semipositive cones}\label{sec4}
For a given proper cone $K$ and matrix $A$, the  $K$-{\it semipositive cone} is defined as $\mathcal{K}_{A,K}=\{x\in K:Ax\in K\}$. Recently many researchers found fascinating results on  geometric properties of  the semipositive cone $K_{A}=\mathcal{K}_{A,\mathbb{R}^{n}_{+}}$. In particular, in \cite{Tsa15}, it is proved that for a square matrix $A$, the cone $K_A$ is a proper polyhedral cone, and later the result was extended for non-square matrices in \cite{TsaS17}. Motivated by their results, the cone $S_A=\{x: Ax\geq 0\}$ was studied in ~\cite{HiaS20}. It is shown that for an invertible matrix $A$, $S_A$ is a simplicial cone, and properties of extremals of $S_A$ were also discussed in ~\cite{HiaS20}.  This drives us to generalize the cone $S_{A}$  to $\mathcal{S}_{A,K}=\{x:Ax\in K\}=A^{-1}(K)$, the pre-image of $K$ under $A$ and, to investigate the properties of extremals of $\mathcal{S}_{A,K}$. It is worth mentioning that $A$ is $K$-semipositive if and only if the intersection of $\mathcal{S}_{A,K}$ and $K$ has a non-empty interior. Also, it may be noted that for any $K$-monotone matrix $A$, the cones  $\mathcal{S}_{A,K}$ and $\mathcal{K}_{A,K}$ are identical. Since  the Lorentz cone is our particular interest of study in this paper,  we emphasis on geometrical properties of  $\mathcal{S}_{A,\mathcal{L}_{+}^{n}}$.  Following definitions are essential for our further study.

%
%
%

\begin{definition}{\rm\cite{BerP94,HiaS20}}\rm \emph{The set generated by $T\subseteq\mathbb{R}^n$}  is the collection of all finite non-negative linear combinations of elements of $T$, and is denoted by $T^{G}$.
A cone $K$ is called \emph{polyhedral cone} if $K=T^{G}$, for some finite set $T$, and a polyhedral cone in $\mathbb{R}^n$ is called \emph{simplicial}, if $T$ contains exactly $n$ number of extremals.  
\end{definition}
One must observe that the generators of polyhedral and simplicial cones are the extremals of the cones~\cite{BerP94}. We now state the theorem related to $S_{A}$  due to \cite{HiaS20} that is generalized in this section.

\begin{theorem}{\rm\cite{HiaS20}}\label{thm4.1} For an invertible matrix $A\in\mathbb{R}^{n\times n}$ the extremals of the simplicial cone $S_A$ have the form $\alpha_{i}A^{-1}e_{i}$ for some $\alpha_{i}>0,~~i=1,\ldots,n$.
\end{theorem} 

Since $e_{i}$'s are the extremals of the simplicial cone $\mathbb{R}^{n}_+$, Theorem~\ref{thm4.1} gives an idea to relate the extremals of $K$ and of $\mathcal{S}_{A,K}$, which is given in the next theorem.
\begin{theorem}
Let $K$ be a proper cone, and  $A\subseteq \mathbb{R}^n$ be an invertible matrix. Then $x$ is an extremal of $K$ if and only if $A^{-1}x$ is an extremal of $\mathcal{S}_{A,K}$.
\end{theorem}
\begin{proof} Assume that $x$ is an extremal of $K$. Let $0\overset{\mathcal{S}_{A,K}}{\leq}y^{'}\overset{\mathcal{S}_{A,K}}{\leq}A^{-1}x$. Write $x^{'}=A^{-1}x$. Since $y^{'}\in \mathcal{S}_{A,K}=A^{-1}K$, choose  $y\in K$, such that $y^{'}=A^{-1}y$, and hence $x^{'}-y^{'}=A^{-1}(x-y)\in A^{-1}K$, that is $x-y\in K$, so that $0\overset{K}{\leq}y\overset{K}{\leq}x$. As $x$ is an extremal of $K$, $x=\alpha y$, for some $\alpha \geq 0$. Therefore, $x^{'}=\alpha y^{'}$. 

Conversely, assume that $x^{'}=A^{-1}x$ is an extremal of $\mathcal{S}_{A,K}$. Let $0\overset{K}{\leq}y\overset{K}{\leq}x$, then $y^{'}=A^{-1}y\in \mathcal{S}_{A,K}$, and $x^{'}-y^{'}=A^{-1}(x-y)\in \mathcal{S}_{A,K}$. Therefore $0\overset{\mathcal{S}_{A,K}}{\leq}y^{'}\overset{\mathcal{S}_{A,K}}{\leq}x^{'}$, so that $x^{'}=\alpha y^{'}$, for some $\alpha \geq 0$. Thus $x=\alpha y$.
\end{proof}
In Theorem~3.6  and 3.10 of \cite{HiaS20}, it is proved that a matrix $A\in\mathbb{R}^{n\times n}$ is minimally semipositive if and only if $K_A=S_{A}$ and $K_A$ is a simplicial cone. In that case, extremals of $K_A$ are the columns of $A^{-1}$. Following two corollaries  generalize these two results for $K$-monotone matrices and follow from the simple fact that $\mathcal{K}_{A,K}=\mathcal{S}_{A,K}$, if $A$ is $K$-monotone.
\begin{corollary}
Let $K$ be a proper cone in $\mathbb{R}^n$, and  $A\in\mathbb{R}^{n\times n}$ is a $K$-monotone matrix. Then $x$ is an extremal of $K$ if and only if $A^{-1}x$ is an extremal of $\mathcal{K}_{A,K}$.
\end{corollary}

\begin{corollary} If $K$ is a simplicial cone in $\mathbb{R}^n$, then so is $\mathcal{S}_{A,K}$. 
\end{corollary}

For a given proper cone $K$,  next theorem furnishes a subclass of $K$-semipositive matrices. More specially, we show that all copositive matrices are $K$-semipositive, if $K$ is self dual.  For our purpose, we now define copositive matrix with respect to a proper cone $K$.
\begin{definition}{\rm \cite{BerP94}} Given a proper cone $K$, a symmetric matrix $A$ is called \emph{copositive over $K$} if $x^TAx\geq 0$ for all $x\in K$.
\end{definition}
If $K$ is self dual cone and $A$ is not semipositive, from results follows from Theorem \ref{one_is_true}.
\begin{proposition}\label{thn4.2}
If $K$ be any self-dual proper cone in $\mathbb{R}^n$, then  copositive  matrices over $K$ are $K$-semipositive.  In particular, symmetric positive definite matrices are $K$-semipositive, for any self dual cone $K$.
\end{proposition}


Finally we discuss few interesting properties of $\mathcal{S}_{A,\mathcal{L}^n_+}$ and $\mathcal{K}_{A,\mathcal{L}^n_+}$. Since  $\mathcal{L}^n_+$ is an ellipsoidal cone, following theorem provides conditions if $\mathcal{S}_{A,\mathcal{L}^n_+}$ and $\mathcal{K}_{A,\mathcal{L}^n_+}$ are ellipsoidal cones. It is well-known from the literature that every ellipsoidal cone is of the form $B\mathcal{L}^n_+$, for some invertible matrix $B$ and vice versa. This fact has been used to describe ellipsoidality of the aforementioned cones.
\begin{theorem}
Let $A\in \mathbb{R}^{n\times n}$. Then we have,
\begin{itemize}
\item [\rm(i)] $\mathcal{S}_{A,\mathcal{L}^n_+}$ is an ellipsoidal cone if and only if $A$ is non-singular.
\item[\rm(ii)] Let $K$ be an ellipsoidal cone. Then there exists an invertible matrix $A$ such that $\mathcal{S}_{A,\mathcal{L}^n_+}=K$. 
\item [\rm(iii)] If $A$ is $\mathcal{L}^n_+$-monotone, then $\mathcal{K}_{A,\mathcal{L}^n_+}$ is an ellipsoidal cone.
\end{itemize}
\end{theorem}
\begin{proof}
(i) Let $\mathcal{S}_{A,\mathcal{L}^n_+}$ be an ellipsoidal cone. If $A$ is singular, there exists $x\neq 0$ such that $Ax=0$, then $x,-x\in \mathcal{S}_{A,\mathcal{L}^n_+}$,which contradicts the pointedness of $\mathcal{S}_{A,\mathcal{L}^n_+}$. Conversely, let $A$ be non-singular. Then $\mathcal{S}_{A,\mathcal{L}^n_+}=A^{-1}\mathcal{L}^n_+$, and hence by Theorem \ref{ellip} it suffices to show that $\mathcal{S}_{A,\mathcal{L}^n_+}$ is a proper cone.

 As $A$ is continuous, and $\mathcal{L}^n_+$ is closed, so $A^{-1}\mathcal{L}^n_+=\mathcal{S}_{A,\mathcal{L}^n_+}$ is closed. Also $A^{-1}\mathcal{L}^n_+$ is convex, since $\mathcal{L}^n_+$ is convex.
 
 Let $x\in A^{-1}\mathcal{L}^n_+\cap \left(-A^{-1}\mathcal{L}^n_+\right)=A^{-1}\left(\mathcal{L}^n_+\cap -\mathcal{L}^n_+ \right)=0$.  Thus $\mathcal{S}_{A,\mathcal{L}^n_+}$ is pointed.
 
 Also, $\mathcal{L}^n_+$ is solid implies that $\mathcal{S}_{A,\mathcal{L}^n_+}$ is solid. Therefore,  $\mathcal{S}_{A,\mathcal{L}^n_+}$ is a proper cone.

(ii) Let $K$ be an ellipsoidal cone, then there exists a non-singular matrix $T$ such that $K=T\mathcal{L}^n_+$. Taking $A=T^{-1}$ we have $\mathcal{S}_{A,\mathcal{L}^n_+}=K$.

(iii) If $A$ is $\mathcal{L}^n_+$-monotone, then by Theorem \ref{monotone}, $A$ is non-singular and $A^{-1}\in \pi(\mathcal{L}^n_+)$, which implies that 
\[\mathcal{K}_{A,\mathcal{L}^n_+}=\mathcal{L}^n_+\cap \mathcal{S}_{A,\mathcal{L}^n_+}= \mathcal{L}^n_+\cap A^{-1}\mathcal{L}^n_+=A^{-1}\mathcal{L}^n_+= \mathcal{S}_{A,\mathcal{L}^n_+}\] 
Thus $\mathcal{K}_{A,\mathcal{L}^n_+}$ is an ellipsoidal cone.
\end{proof}

\begin{corollary}
$A$ is $\mathcal{L}^n_+$-monotone if and only if $A$ is invertible and  $\mathcal{K}_{A,\mathcal{L}^n_+}=\mathcal{S}_{A,\mathcal{L}^n_+}$  is an ellipsoidal cone.
\end{corollary}
\section{Conclusion}\label{sec5} In this paper, we have studied algebraic properties of $\mathcal{L}^n_+$-semipositive matrices. It is observed that the $n$-th row and $n$-th column play important roles to provide necessary (or sufficient) conditions for semipositivity of matrices over $\mathcal{L}^n_+$. We further noticed that diagonal and orthogonal $\mathcal{L}^n_+$-semipositive matrices can be characterized with the $(n,n)$-th entry of the matrix. Also, we have found sufficient conditions for lower triangular matrices to be $\mathcal{L}^n_+$-semipositive matrices. Given a non-singular matrix $A$, and a proper cone $K$, we have examined the geometric properties of the semipositive cone $\mathcal{K}_{A,K}$ and of the cone $\mathcal{S}_{A,K}$ in terms of their extremals. In particular, it is found that extremals of $K$ and $\mathcal{S}_{A,K}$ are connected via $A^{-1}$. Lastly, we have obtained sufficient conditions for the two specific cones $\mathcal{S}_{A,\mathcal{L}^n_+}$ and $\mathcal{S}_{A,\mathcal{L}^n_+}$ to be ellipsoidal. 

\section*{Acknowledgement} The authors are grateful to the anonymous reviewers for their positive criticism and valuable suggestions on the work, that significantly improved the quality of the paper.

\end{document}